\documentclass[11pt]{article}
\usepackage{fullpage}
\usepackage{amssymb,calc,enumerate,booktabs}
\usepackage{amsmath,mathrsfs,yfonts}
\usepackage{amsthm}
\usepackage{etoolbox,url,textpos}
\usepackage[12hr]{datetime}
\usepackage[ruled,linesnumbered]{algorithm2e}
\usepackage{units}
\usepackage{tikz,adjustbox}
\usepackage{pgfplots}
\usepackage{pdflscape}
\usepackage[affil-sl]{authblk}
\usepackage{graphicx,eso-pic,xcolor}
\newcommand{\cl}[1]{\mathscr{#1}}

\newcommand{\bl}[1]{\mathbb{#1}}
\newcommand{\bd}[1]{\bf{#1}}
\usetikzlibrary{positioning,shapes,shadows,arrows}
\makeatletter
\def\@maketitle{%
  \newpage
  \null
  \vskip 2em%
  \begin{center}%
  \let \footnote \thanks
    {\Large\bfseries \@title \par}%
    \vskip 1.5em%
    {\normalsize
      \lineskip .5em%
      \begin{tabular}[t]{c}%
        \@author
      \end{tabular}\par}%
    \vskip 1em%
    {\normalsize \@date}%
  \end{center}%
  \par
  \vskip 1.5em}
\makeatother
\newtheorem{theorem}{Theorem}

\newtheorem{corollary}[theorem]{Corollary}

 \newtheoremstyle{component}{}{}{}{}{\bfseries\itshape}{:}{.5em}{\thmnote{#3}#1}
    \theoremstyle{component}
    
\makeatletter
\patchcmd{\@algocf@start}{%
  \begin{lrbox}{\algocf@algobox}%
}{%
  \rule{0.1\textwidth}{\z@}%
  \begin{lrbox}{\algocf@algobox}%
  \begin{minipage}{0.8\textwidth}%
}{}{}
\patchcmd{\@algocf@finish}{%
  \end{lrbox}%
}{%
  \end{minipage}%
  \end{lrbox}%
}{}{}
\newenvironment{mathprog}[1][]%
{\par\bigskip\noindent\ignorespaces\begin{tabular*}{\textwidth}{@{} p{0.35\textwidth} @{\hspace{0.5em}} l} #1}%
{\end{tabular*}\\\smallskip\par\noindent\ignorespacesafterend}

\newcommand{\progline}[2][]{\hfill #1 & \begin{math}\displaystyle #2 \end{math}\\}

\makeatother

\title{\LARGE Average value of solutions for the bipartite boolean quadratic programs and rounding algorithms\footnote{This research work was supported by an NSERC  Discovery grant and an NSERC discovery accelerator grant awarded to Abraham P Punnen. }}
\author {Abraham P. Punnen\thanks{apunnen@sfu.ca}}
\author {Piyashat Sripratak\thanks{psriprat@sfu.ca}}
\author {Daniel Karapetyan\thanks{daniel.karapetyan@gmail.com}}
\affil{Department of Mathematics, Simon Fraser University Surrey, 250-13450 102nd AV, Surrey, British Columbia, V3T 0A3, Canada}

\date{}

\begin{document}

\maketitle

\begin{abstract}
We consider domination analysis of approximation algorithms for the bipartite boolean quadratic programming problem (BBQP) with $m+n$ variables.  A closed form formula is developed to compute the average objective function value $\cl{A}$ of all solutions in $O(mn)$ time. However, computing the median objective function value of the solutions is shown to be NP-hard. Also, we show that any solution with objective function value no worse than $\cl{A}$  dominates at least $2^{m+n-2}$ solutions and this bound is the best possible. Further, we show that such a solution can be identified in $O(mn)$ time and hence the dominance ratio of this algorithm is at least $\frac{1}{4}$. We then show that for any fixed rational number $\alpha > 1$,  no polynomial time approximation algorithm exists for BBQP with dominance ratio larger than $1-2^{\frac{(1-\alpha)}{\alpha}(m+n)}$, unless P=NP.  We then analyze some powerful local search algorithms and show that they can get trapped at a local maximum with objective function value less than $\cl{A}$. One of our approximation algorithms has an interesting rounding property  which provides a data dependent lower bound on the optimal objective function value. A new integer programming formulation of BBQP is also given and  computational results with our rounding algorithms are reported.\\

\end{abstract}

\section{Introduction}
The {\it   bipartite boolean quadratic programming problem} (BBQP) is to
\begin{mathprog}
\progline[Maximize]{f(x,y) = x^TQy + cx+dy}
\progline[subject to]{x \in \{0,1\}^m, y\in \{0,1\}^n}
\end{mathprog}
where $Q = (q_{ij})$ is an $m\times n$ matrix, $c=(c_1,c_2, \ldots ,c_m)$ is a row vector in $R^m$, and $d=(d_1,d_2,\ldots ,d_n)$ is a row vector in $R^n$\@. Without loss of generality, we assume that $m \leq n$.

BBQP has applications in data mining, clustering and bioinformatics~\cite{tanay}, approximating a matrix by a rank-one binary matrix~\cite{gills,shen}, mining discrete patterns in binary data~\cite{lu,shen}, solving fundamental graph theoretic optimization problems such as maximum weight biclique~\cite{amb,tan}, maximum weight cut problem on a bipartite graph~\cite{p1},  maximum weight induced subgraph of a bipartite graph~\cite{p1}, and computing approximations to the cut-norm of a matrix~\cite{alon}.

BBQP is closely related to the well-studied {\it boolean quadratic programming problem} (BQP)~\cite{b1,gl3,w1}:
\begin{mathprog}
\progline[Maximize]{ f(x) = x^TQ^{\prime}x + c^{\prime}x}
\progline[subject to]{ x \in \{0,1\}^n,}
\end{mathprog}
where $Q^{\prime}$ is an $n\times n$ matrix and $c{^{\prime}}$ is a row vector in $R^n$.
BBQP can be formulated as a BQP with $n+m$ variables~\cite{p1} and hence the resulting cost matrix will have dimension $(n+m)\times (n+m)$. This increase in problem size is not desirable especially for large scale problems. On the other hand, we can formulate BQP  as a
BBQP by choosing
 \begin{equation}\label{equa}Q=Q^{\prime}+2MI,\; c=\frac{1}{2}c^{\prime}-Me \mbox{ and } d=\frac{1}{2}c^{\prime}-Me,\end{equation}
 where $I$ is an $n\times n$ identity matrix, $e\in R^n$ is an all one row vector and $M$ is a very large number~\cite{p1}. Thus, BBQP is a proper generalization of BQP which makes the study of BBQP further interesting. An instance of BBQP is completely defined by the matrix $Q$ and vectors $c$ and $d$ and hence it is represented by $\cl{P}(Q,c,d)$. Thus, $\cl{P}(Q,0,0)$ represents a BBQP with no terms  $cx$ or $dy$ in the objective function. Such an instance is referred to as a {\it homogeneous} BBQP. Relationships between BBQP and its homogeneous version are considered in~\cite{p1}.

BBQP is trivial if the entries of $Q$,$c$ and $d$ are either all positive or all negative.  BBQP is known to be NP-hard~\cite{p2} since the maximum weight biclique problem (MWBP) is a special case of it. Approximation hardness results for MWBP are established by Ambuhl et al.~\cite{amb} and Tan~\cite{tan}. Performance ratio for approximation algorithms for some special cases of BBQP are discussed by Alon and Naor~\cite{alon} and Raghavendra and Steurer~\cite{rs}. Results extensive experimental analysis of algorithms for BBQP are reported by Karapetyan and Punnen~\cite{kp} and Glover et al.~\cite{gl4}.  Punnen, Sripratak, and Karapetyan~\cite{p1} studied BBQP and identified various polynomially solvable special cases. Various classes of valid inequalities and facet defining inequalities for the polytope associated with BBQP are obtained by Sripratak and Punnen~\cite{sp}.

Worst case analysis of approximation algorithms (heuristics) are carried out normally through the measure of performance ratio~\cite{vv}. Other important measures include differential ratio~\cite{dm}, dominance ratio~\cite{gp,h1}, dominance number~\cite{z1,gp}, comparison to average value of solutions~\cite{an1,t1,sn3,sd2,rb1} etc. Our focus in this paper is on domination analysis and average value based analysis of approximation algorithms for BBQP. Berend et al.~\cite{b1} eloquently argues the importance of domination analysis in the study of approximation algorithms.

 Let $\cl{F}$ be the family of all solutions of BBQP and it is easy to see that $|\cl{F}|=2^{m+n}$\@.  The average objective function value $\cl{A}(Q,c,d)$ of all the solutions  of BBQP  is given by $\cl{A}(Q,c,d)=2^{-(m+n)}\sum_{(x,y)\in \cl{F}}f(x,y)$. The idea of comparing a heuristic solution to the average objective function value of all the solutions as a measure of heuristic quality for combinatorial optimization problems was originated in the Russian literature in the early 1970s. Most of these studies are focussed on the traveling salesman problem and the assignment problem (e.g. Rublineckii~\cite{rb1}, Minina and Perekrest~\cite{mn1}, Vizing~\cite{v1}, Sarvanov and Doroshko~\cite{sn1,sd2}). In the western literature,  Gutin and Yeo~\cite{gutin2}, Grover~\cite{gr1}, Punnen et al.~\cite{p2}, Punnen and Kabadi~\cite{p3}, Deneko and Woeginger~\cite{x1} studied the traveling salesman problem and identified heuristics that guarantee a solution with objective function value no worse than the average value of all tours. Such a solution has interesting domination properties and hence the approach is also relevant in dominance analysis of heuristics. For recent developments on domination analysis, we refer to the excellent research papers~\cite{a1,b1,gp,gutin3}. Gutin and Yeo~\cite{gutin2}, Sarvanov~\cite{sn3}, and Angel et al.~\cite{an1} studied heuristics for the quadratic assignment problem with performance guarantee in terms of average value of solutions. Similar analysis for the three-dimensional assignment problem was considered by Sarvanov~\cite{sn3}, for the Maximum clique problem by Bendall and Margot~\cite{ben2}, and for the satisfiability problem by Twitto~\cite{t1}. Berend et al.~\cite{b1} considered dominance analysis by including infeasible solutions. Other problems studied from the point of view of dominance analysis and average value based analysis include graph bipartition,  variations of maximum clique and independent set problems~\cite{an1,gr1} and the subset-sum problem~\cite{b1}. For information on dominance results and linkages with the development of heuristic algorithms based on very large scale neighborhood search, we refer to~\cite{a14}.

A solution with objective function value no worse than the average value of a solution with high probability can be  obtained by repeated random sampling. However, it should be pointed out that even algorithms that performs well in practice could produce solutions with objective function value inferior to the average value of a solution. We observed this particularly in the case of the BBQP. Thus, a worst case performance  of a heuristic that guarantees a solution with objective function value no worse than the average value of a solution is a useful measure to be included when studying worst case behavior of heuristic algorithms for combinatorial optimization problems.

Let $(x,y),(x^0,y^0)\in \cl{F}$. Then $(x^0,y^0)$ \emph{dominates} $(x,y)$ if $f(x,y)\leq f(x^0,y^0)$. Let $\Gamma$ be a heuristic algorithm for BBQP that produces a solution $(x^{\Gamma},y^{\Gamma})$. Define $\cl{G}^{\Gamma}=\{(x,y)\in \cl{F} : f(x,y)\leq f(x^{\Gamma},y^{\Gamma})\}$. Let $I$ be the collection of all instances of BBQP. Then the \emph{dominance number} and \emph{dominance ratio} of $\Gamma$ are defined respectively as
$$ \inf_{\cl{P}(Q,c,d)\in I}\left|\cl{G}^{\Gamma}\right|\mbox{~ and } \inf_{\cl{P}(Q,c,d)\in I}\left\{\frac{\left |\cl{G}^{\Gamma}\right |}{\left|\cl{F}\right|}\right\}.$$ The concept of dominance ratio in the analysis of heuristics was proposed by Glover and Punnen~\cite{gp}. Prior to this work, Zemel~\cite{z1} considered different measures to analyze heuristic algorithms one of which is equivalent to the dominance number. Hassin and Kuller~\cite{h1} also considered similar measures in analyzing heuristic algorithms.

In this paper we obtain a closed form formula to compute $\cl{A}(Q,c,d)$ in $O(mn)$ time. We also show that any solution to BBQP with objective function value no less than $\cl{A}(Q,c,d)$ dominates $2^{m+n-2}$ solutions. Such a solution is called \emph{no worse than average} solution.  Two algorithms of complexity $O(mn)$ are developed to compute no-worse than average solutions.  Thus, the dominance ratio of these algorithms is at least $\frac{1}{4}$. One of these algorithms have interesting rounding property which provides data dependent lower bounds. The problem of computing a solution with objective function value no worse than the median of the objective function values of all solutions is shown to be NP-hard. Further, we show that, unless P=NP, for any fixed rational number $\alpha > 1$,  no polynomial time approximation algorithm exists for BBQP with dominance ratio larger than $1-2^{\frac{(1-\alpha)}{\alpha}(m+n)}$. We also analyze some very powerful local search algorithms and show that, in worst case, such algorithms could get trapped at a locally optimal solution with objective function value less than $\cl{A}(Q,c,d)$. Finally we provide a new integer programming formulation of BBQP and the resulting LP relaxation solution could be used to initiate our rounding algorithms. Computational results are also provided using the rounding algorithms which establish that the algorithms are good candidates to obtain very fast starting solutions in complex metaheuristic algorithms.

Through out this paper, we use the following notations and naming conventions. We denote $\cl{M}=\{1,2,\ldots ,m\}$ and $\cl{N}=\{1,2,\ldots ,n\}$. The $i$th component of a vector is represented simply by adding the subscript $i$ to the name of the vector. For example, the $i$th component of the vector $x^*$ is $x^*_i$. The set $\{0,1\}^n$ is denoted by $\bl{B}^n$ and $[0,1]^n$ is denoted by $\bl{U}^n$ for any positive integer $n$. For any positive integer $m$, an $m$-vector of all 1's is denoted by ${\bd{1}}^m$  and an $m$-vector of all 0's is denoted by ${\bd{0}}^m$.

\section{Average value of solutions and dominance properties}

Note that there are $2^m$ candidate solutions for $x$ and $2^n$ candidate solutions for $y$. Then the solutions in the family $\cl{F}$ can be enumerated as  $\cl{F}=\{(x^k,y^{\ell}) : k=1,2,\ldots ,2^m, \ell=1,2,\ldots ,2^n\}$.  The next theorem gives a closed form expression to compute $\cl{A}(Q,c,d)$ in $O(mn)$ time.
\begin{theorem}$\label{th1}{\displaystyle \cl{A}(Q,c,d)=\dfrac{1}{4}\sum_{i\in \cl{M}}\sum_{j\in \cl{N}}q_{ij}+\dfrac{1}{2}\sum_{i\in \cl{M}}c_i+\dfrac{1}{2}\sum_{j\in \cl{N}}d_j}$.\end{theorem}
\begin{proof}
Let $\eta=2^m$ and $\nu=2^n$. Then
\begin{align*}
\cl{A}(Q,c,d)&=\dfrac{1}{2^{m+n}}\sum_{k=1}^{\eta}\sum_{\ell=1}^{\nu}f(x^k,y^{\ell})\\
&=\dfrac{1}{2^{m+n}}\sum_{k=1}^{\eta}\sum_{\ell=1}^{\nu}\left(\sum_{i\in \cl{M}}\sum_{j\in \cl{N}}q_{ij}x^k_iy^{\ell}_j +\sum_{i\in \cl{M}}c_ix^k_i+\sum_{j\in \cl{N}}d_jy^{\ell}_j\right )\\
&=\dfrac{1}{2^{m+n}}\left(\sum_{i\in \cl{M}}\sum_{j\in \cl{N}}q_{ij}\sum_{k=1}^{\eta}x^k_i\sum_{\ell=1}^{\nu}y^{\ell}_j +\sum_{i\in \cl{M}}c_i\sum_{\ell=1}^{\nu}\sum_{k=1}^{\eta}x^k_i+\sum_{j\in \cl{N}}d_j \sum_{k=1}^{\eta}\sum_{\ell=1}^{\nu}y^{\ell}_j\right )\\
&=\dfrac{1}{2^{m+n}}\left(2^{m-1}2^{n-1}\sum_{i\in \cl{M}}\sum_{j\in \cl{N}}q_{ij} +\nu 2^{m-1}\sum_{i\in \cl{M}}c_i+\eta 2^{n-1}\sum_{j\in \cl{N}}d_j\right )\\
&=\dfrac{1}{4}\sum_{i\in \cl{M}}\sum_{j\in \cl{N}}q_{ij}+\dfrac{1}{2}\sum_{i\in \cl{M}}c_i+\dfrac{1}{2}\sum_{j\in \cl{N}}d_j.
\end{align*}
\end{proof}
If either $x=0$ or $y=0$, then $x^TQy=0$ and such an  $(x,y)$ is  called a {\it trivial solution}. Note that $f(x,y)$ need not be equal to zero for trivial solutions. All  remaining solutions are called {\it nontrivial solutions}. Maximizing $f(x,y)$ over trivial solutions is straightforward and thus one can restrict attention to non-trivial solutions only. The number of nontrivial solutions is $(2^m-1)(2^n-1)$. Let $\bar{A}(Q,c,d)$ denote the average value of all nontrivial solutions for $\cl{P}(Q,c,d)$.
\begin{corollary}\label{cor1}
${\displaystyle \bar{A}(Q,c,d)=\dfrac{2^{m-1}2^{n-1}}{(2^{m}-1)(2^n-1)}\left(\sum_{i\in \cl{M}}\sum_{j\in \cl{N}}q_{ij} +\dfrac{2^n-1}{ 2^{n-1}}\sum_{i\in \cl{M}}c_i+\dfrac{2^m-1}{ 2^{m-1}}\sum_{j\in \cl{N}}d_j\right )}$. Further, $\bar{A}(Q,{\bd{0}}^m,{\bd{0}}^n)=\dfrac{2^{m+n}}{(2^{m}-1)(2^n-1)}\cl{A}(Q,{\bd{0}}^m,{\bd{0}}^n)$ and ${\displaystyle \lim_{n\rightarrow \infty}\lim_{m\rightarrow \infty}\bar{A}(Q,c,d)=\cl{A}(Q,c,d)}$.
\end{corollary}
From Corollary~\ref{cor1}, the asymptotic behavior of $\bar{A}(Q,c,d)$ and $\cl{A}(Q,c,d)$ are the same. Thus, hereafter we focus our attention on $\cl{A}(Q,c,d)$ only.

Let $\cl{G}=\{(x,y) : x\in \{0,1\}^m, y\in \{0,1\}^n, f(x,y) \leq \cl{A}(Q,c,d)\}$. Thus, $\cl{G}$ consists of all solutions of BBQP that are no worse than average.

\begin{theorem}\label{dth1} $|\cl{G}|\geq 2^{m+n-2} $.\end{theorem}
\begin{proof}For any solution $(x,y)\in \cl{F}$, let $P(x,y)=\{(x,y),(x,{\bd{1}}^n-y), ({\bd{1}}^m-x,y), ({\bd{1}}^m-x,{\bd{1}}^n-y)\}$. It can be verified that $P(x,y)=P(x,{\bd{1}}^n-y)=P({\bd{1}}^m-x,y)=P({\bd{1}}^m-x,{\bd{1}}^n-y)$ and $P(x, y) \neq P(x', y')$ if $x' \notin \{ x, {\bd{1}}^m
- x \}$ or $y' \notin \{ y, {\bd{1}}^n - y \}$.
 Thus, we can partition the solution space $\cl{F}$ into $\frac{1}{4}2^{m+n}=2^{m+n-2}$ disjoint sets $P(x^k,y^k)$,  $k=1,2,\ldots ,2^{m+n-2}=\omega$, say. Note that
\begin{align}\nonumber
f(x,y)+f({\bd{1}}^m-x,y)+f(x,{\bd{1}}^n-y)+f({\bd{1}}^m-x,{\bd{1}}^n-y)& = \sum_{i\in \cl{M}}\sum_{j\in \cl{N}}q_{ij}+2\sum_{i\in \cl{M}}c_i+2\sum_{j\in \cl{N}}d_j\\& = 4\cl{A}(Q,c,d). \label{doeq1}
\end{align}
From equation (\ref{doeq1}), it follows immediately that $$\min\left\{f(x,y), f({\bd{1^m}}-x,y), f(x,{\bd{1}}^n-y),f({\bd{1}}^m-x,{\bd{1}}^n-y)\right\}\leq \cl{A}(Q,c,d).$$ Thus, from each $P(x^k,y^k), k=1,2,\ldots \omega$, choose a solution with smallest objective function value to form the set $D_1$. By construction, $f(x,y)\leq \cl{A}(Q,c,d)$ for all $(x,y)\in D_1$.  Since $|D_1|=2^{m+n-2}$, the result follows.
\end{proof}

The lower bound on $\cl{G}$ established in Theorem~\ref{dth1} is tight. To see this, consider the matrix $Q$ defined by
\begin{align*}
q_{ij} = \begin{cases}
-1 & \text{if } i=m,j=n, \\
0 & \text{otherwise.}
\end{cases}
\end{align*}
and choose $c$ and $d$ as zero vectors in $R^m$ and $R^n$, respectively. Then $\cl{A}(Q,c,d)=\frac{-1}{4}$ and the set of solutions $(x,y)$ with $f(x,y) \leq \cl{A}(Q,c,d)$ is precisely $\cl{G}=\{(x,y) | x_m=y_n=1\}$. Clearly, $|\cl{G}|=2^{m+n-2}$ and hence the bound obtained in Theorem~\ref{dth1} is the best possible. \\

Theorem~\ref{dth1}  establishes that any algorithm that guarantees a solution with objective function value no worse than $\cl{A}(Q,c,d)$ dominates $2^{m+n-2}$ solutions of the BBQP $\cl{P}(Q,c,d)$. In other words, the domination ratio of such an algorithm is at least $1/4$. Our next theorem establishes an upper bound on the dominance number of any polynomial time approximation algorithm for BBQP.

\begin{theorem}\label{cth1}Unless P=NP, no polynomial time algorithm for BBQP can have dominance number more than $2^{m+n}-2^{\lfloor \frac{m+n}{\alpha}\rfloor}$ for any fixed rational number $\alpha > 1$.\end{theorem}
\begin{proof}We show that a polynomial time algorithm $\Omega$ for BBQP with dominance number at least $2^{m+n}-2^{\lfloor \frac{m+n}{\alpha}\rfloor}+1$ can be used to compute an optimal solution to BBQP. Without loss of generality, assume $\alpha=\frac{a}{b}$ where $a$ and $b$ are positive relatively prime integers with $a > b$. Consider an instance $\cl{P}(Q,c,d)$ of BBQP. Let $Q^*=(q^*_{ij})$ be an $ab m\times ab n$ matrix where
 $$
q^*_{ij} = \begin{cases}
q_{ij} & \text{if  $i\in \cl{M}$ and $j\in \cl{N}$}, \\
0 & \text{otherwise.}
\end{cases}
$$
Likewise, let $c^*$ and $d^*$ be vectors in $R^{ab m}$ and $R^{ab n}$ such that
\begin{center}
\begin{tabular}{ccc}
$
c^*_i = \begin{cases}
c_i & \text{if } \displaystyle{i\in \cl{M}}, \\
0 & \text{otherwise}
\end{cases}
$ & and &
$
d^*_j = \begin{cases}
d_j & \text{if } \displaystyle{j\in \cl{N}}, \\
0 & \text{otherwise.}
\end{cases}
$
\end{tabular}
\end{center}
It is easy to verify that from any optimal solution to the BBQP instance $\cl{P}(Q^*,c^*,d^*)$  an optimal solution to $\cl{P}(Q,c,d)$ can be recovered.  The total number of solutions of $\cl{P}(Q^*,c^*,d^*)$ is $2^{ab(m+n)}$ of which at least $ 2^{ab(m+n)-(m+n)}$ are optimal. So the maximum number of non-optimal solutions is $2^{ab(m+n)}-2^{ab(m+n)-(m+n)}$. Solve the BBQP instance $\cl{P}(Q^*,c^*,d^*)$ using $\Omega$ and let  $(x^*,y^*)$ be the resulting solution. By hypothesis, the objective function value of $(x^*,y^*)$ is not worse than that of at least $2^{ab(m+n)} - 2^{\frac{ab(m+n)}{a/b}}+1=2^{ab(m+n)} - 2^{b^2(m+n)}+1$ solutions. Since  $a > b$, we have $2^{ab(m+n)} - 2^{b^2(m+n)}+1 > 2^{ab(m+n)}-2^{ab(m+n)-(m+n)}$. Thus, $(x^*,y^*)$ must be optimal for $\cl{P}(Q^*,c^*,d^*)$. From $(x^*,y^*)$, an optimal solution to $\cl{P}(Q,c,d)$ can be recovered by simply taking the first $m$ components of $x^*$ and first $n$ components of $y^*$. The result now follows from NP-completeness of BBQP.
\end{proof}
Theorem~\ref{cth1} implies that unless P=NP, no polynomial time approximation algorithm for BBQP can have dominance ratio more than $1-2^{\frac{(1-\alpha)}{\alpha}(m+n)}$ for any fixed rational number $\alpha > 1$.

Although we have a closed form formula for  computing the average value of all solutions to BBQP, we now show that computing the median value of all solutions is NP-hard.

  Since $|\cl{F}|$ is even,  there are two values of median, say $\theta_1$ and $\theta_2$, where $\theta_1\leq \theta_2$.  A median finding algorithm could simply produce $\theta_1$ or $\theta_2$, but we may not know precisely, the output is either $\theta_1$ or $\theta_2$.
\begin{theorem}Computing a median of the objective function values of BBQP is NP-hard.\end{theorem}
\begin{proof}
Suppose we have a polynomial time algorithm to compute a median of the objective function values of BBQP. We will show that this algorithm can be used to solve the PARTITION problem, which is defined as follows: Given $n$ positive integers $a_1, a_2,\ldots ,a_n$, determine if there exists a partition $S_1$ and $S_2$ of $N=\{1,2,\ldots ,n\}$ such that $\sum_{j\in S_1}a_j=\sum_{j\in S_2}a_j$. From an instance of PARTITION, construct an instance of BBQP as follows: Choose $c$ as the zero vector, $d_j=a_j$ for $j=1,2,\ldots ,n$. Choose $\cl{M}=\{1,2\}$. Define $q_{1j}=a_j\epsilon$ and $q_{2j}=-a_j\epsilon$, where $\epsilon$ is a very small positive number. For each subset $H$ of $N$, let $y^H$ be its characteristic vector, i.e. $y^H\in \bl{B}^n$ and $y^H_j=1$ if and only if $j\in H$. For each choice of $H$, we can associate four choices for $x$ as $x=(0,0),~x=(1,0),~x=(0,1)$ or $x=(1,1)$. Thus, for each $H$, we get the following solutions $F_H=\{((0,0),y^H),~ ((1,0),y^H),~ ((0,1),y^H),~ ((1,1),y^H)\}$. Now,
$f((0,0),y^H)=\sum_{j\in H}a_j, f((1,0),y^H)=(1+\epsilon)\sum_{j\in H}a_j, f((0,1),y^H)=(1-\epsilon)\sum_{j\in H}a_j$ and $f((1,1),y^H)=\sum_{j\in H}a_j$. Thus, $f((0,1),y^H)<f((0,0),y^H)=f((1,1),y^H)<f((1,0),y^H)$. There are $2^n$ choices for $H$ and hence there are $2^{n+2}$ different solutions for the BBQP constructed.
For each subset $H$ of $N$, let $g(H)=\sum_{j\in H}a_j$ and $G=\{g(H) : H\subseteq N\}$. We first observe that $G$ has two medians and these median values are the same and equal to $\frac{1}{2}\sum_{j\in N}a_j$ if and only if $N$ has the required partition. This follows from the fact that for any $H\subseteq N$, either $g(H)\leq\frac{1}{2}\sum_{j\in N}a_j\leq g(N\setminus H)$ or $g(N\setminus H)\leq \frac{1}{2}\sum_{j\in N}a_j\leq g(H)$.

 Let $\sigma_1 < \sigma_2 <\cdots < \sigma_{\kappa}$ be an ascending arrangement of distinct  $g(H)$, $H\subseteq N$ and let $W_k=\{H \subseteq N : g(H)=\sigma_k\}$. Note that $|W_k|=|W_{\kappa+1-k}|$ and $\sigma_k+\sigma_{\kappa+1-k}= \sum_{j\in N}a_j$. Thus, the required partition exists if and only if median of $\{\sigma_1,\sigma_2,\ldots ,\sigma_{\kappa}\}=\frac{1}{2}\sum_{j\in N}a_j$. Consider an ascending arrangement of $f(x,y)$ for all solutions $(x,y)$ of the BBQP constructed. This can be grouped as blocks of values $B_1 < B_2 < \cdots < B_{\kappa}$ where the block $B_k$ has the structure
 \begin{align*}
 \overbrace{\sigma_k(1\!-\!\epsilon)=\sigma_k(1\!-\!\epsilon)=\cdots =\sigma_k(1\!-\!\epsilon)}^{\text{repeated $|W_k|$ times}} < \underbrace{\sigma_k=\sigma_k=\cdots=\sigma_k }_{\text{repeated $2|W_k|$ times}} < \overbrace{\sigma_k(1\!+\!\epsilon)=\sigma_k(1\!+\!\epsilon)=\cdots =\sigma_k(1\!+\!\epsilon)}^{\text{repeated $|W_k|$ times}}
 \end{align*}
 \noindent for $k=1,2,\ldots ,\kappa$. Thus, median of $\{\sigma_1,\sigma_2,\ldots ,\sigma_{\kappa}\}$ is the same as median of the objective function values of BBQP. Thus, the required partition exists if and only both the median values are the same and equal to $\frac{1}{2}\sum_{j\in N}a_j$. The proof now follows from the NP-completeness of PARTITION.
\end{proof}
It may be noted that the above theorem does not rule out the possibility of a polynomial time algorithm with dominance number $2^{m+n-1}$\@.

%
%

\section{Average value of solutions and local search}
In this section we consider two natural local search heuristics for BBQP and show that the solution produced could have objective function value worse than $\cl{A}(Q,c,d)$. One of the popular heuristics for BBQP is the {\it alternating algorithm} proposed by many authors~\cite{lu,gills,kp}. The algorithm starts with a candidate solution $x^0$ and try to choose an optimal $y^0$. Then fix $y^0$ and tries to find the best candidate for $x$, say $x^1$ yielding a solution $(x^1,y^0)$. These operations can be carried out using the formulas

\begin{center}
\begin{tabular}{ccc}
$
y^0_j = \begin{cases}
1 & \text{if } \displaystyle{\sum_{i\in \cl{M}} q_{ij} x^0_i + d_j > 0}, \\
0 & \text{otherwise,}
\end{cases}
$ & and &
$
x^1_i = \begin{cases}
1 & \text{if } \displaystyle{\sum_{j\in \cl{N}} q_{ij} y^0_j + c_i > 0}, \\
0 & \text{otherwise,}
\end{cases}
$
\end{tabular}
\end{center}

 Now fix $x=x^1$ and try to choose the best $y=y^1$ and the process is continued until no improvement is possible by fixing either $x$ variables or $y$ variables. The algorithm terminates when a locally optimal solution is reached. From experimental analysis~\cite{kp}, it is known that this algorithm produces reasonably good solutions on average. To the best of our knowledge, worst-case behavior of this algorithm has not been investigated.

\begin{theorem}The objective function value of a locally optimal solution produced by the alternating algorithm could be arbitrarily bad and could be worse than $\cl{A}(Q,c,d)$.
\end{theorem}
\begin{proof}
Choose $m=n$, $c=d=0$, set $q_{11} =1, q_{nn}=M$ and $q_{ij}=0 $ for all other combinations of $i$ and $j$. Choose the starting solution $x^0_1=1$ and $x^0_i=0$ for $i\neq 1$.  The algorithm will choose $y^0_1=1$ and $y^0_j=0$ for $j\neq 1$. Now $(x^0,y^0)$ will be locally optimal solution with objective function value 1, but the optimal objective function value is $M+1$ for any $M > 0$. The average cost of a solution is $\frac{M+1}{4}$. Thus, for $M > 3$ the solution produced by the alternating algorithm is less than $\cl{A}(Q,c,d)$. In fact, by choosing $M$ large, the solution can be made arbitrarily bad.
\end{proof}

Despite this example, it is easy to see that a solution produced by the alternating algorithm dominates at least $2^m+2^n-1$ solutions.

Let us now consider a more general neighborhood, which is a variation of the $k$-exchange neighborhood studied for various combinatorial optimization problems. For any $(x^0,y^0)\in \cl{F}$, let $\bl{N}^{hk}$ be the set of solutions in $\cl{F}$ obtained by flipping at most $h$ components of $x^0$ and at most $k$ components of $y^0$. If $h=m$ or $k=n$, we ignore ``,'' in the definition of $\bl{N}^{h,k}$. Note that $|\bl{N}^{h,k}|=\left(\sum_{j=0}^h{m\choose j}\right )\left(\sum_{i=0}^k {n\choose i}\right )$, and the best solution in this neighborhood can  be identified in polynomial time for fixed $h$ and $k$. A more powerful neighborhood is $\bl{N}^{\alpha}=\bl{N}^{m\alpha}\cup \bl{N}^{\alpha n}$ and $|\bl{N}^{\alpha}|=2^m\sum_{j=0}^\alpha{n\choose j}+2^n\sum_{i=0}^\alpha {m\choose i}-\sum_{i=0}^\alpha {m\choose i}\sum_{j=0}^\alpha{n\choose j}$. Again, this neighborhood can also be searched for an improving solution in polynomial time for fixed $\alpha$~\cite{kp}. It may be noted that a solution produced by the alternating algorithm is locally optimal with respect to the neighborhood $\bl{N}^0=\bl{N}^{m0}\cup \bl{N}^{0n}$. Glover et al.~\cite{gl4} considered the neighborhoods $\bl{N}^1$, $\bl{N}^2$, and $\bl{N}^{1,1}$. They provided fast and efficient algorithms for exploring these neighborhoods supported by detailed computational analysis. They also considered tabu search algorithms using these neighborhoods in a hybrid form. Computational results with these algorithms provided very high quality solutions, improving several benchmark instances. Nonetheless, our next theorem shows that even such very powerful local search algorithms could provide solutions with objective function values that are inferior to $\cl{A}(Q,c,d)$ even if we allow $\alpha$ to be a function of $n$.

\begin{theorem}\label{loc1}A locally optimal solution to BBQP with respect to the neighborhood $\bl{N}^{\alpha}=\bl{N}^{m\alpha}\cup \bl{N}^{\alpha n}$ could be worse than average for any $\alpha \leq \left\lfloor\frac{n}{5}\right \rfloor$.
\end{theorem}
\begin{proof}Consider the matrix $Q$ defined as
\begin{align}\label{peq1}
q_{ij} = \begin{cases}
\lambda & \text{if } i=m, j=n, \\
-1 & \text{if $i=m$ or $j=n$ but $(i,j)\neq (m,n)$},\\
a & \text{otherwise}.
\end{cases}
\end{align}
and choose $c$ and $d$ as zero vectors. Without loss of generality, we assume $m=n$. Otherwise, we can extend the matrix $Q$ into an $n\times n$ matrix by adding $n-m$ rows of zeros and extending the vector $c$ into an $n$-vector by making the last $n-m$ entries zeros. We also assume that $n$ is a multiple of 5. Consider the solution $(x^0,y^0)$ where $x^0_n=y^0_n=1$ and all other components are zero. Also, let $\alpha=\frac{n}{5}$ and assume $n\geq 6$. \\

Let $N^r_x(0)$ be the set of all $x\in \bl{B}^n$ with $x_n=0$ obtained by flipping exactly $r$ entries of $x^0$ and  $N^r_x(1)$ be the set of all $x\in \bl{B}^n$ with $x_n=1$ obtained by flipping exactly $r$ entries of $x^0$. Define $N^r_y(0)$ and $N^r_y(0)$ analogously. Note that for any $(x,y)\in \bl{N}^{\alpha n}$, $x\in N^r_x(0)\cup N^r_x(1)$ and  $y\in N^s_y(0)\cup N^s_y(1)$ for some $0\leq r\leq \alpha, 0\leq s\leq n$ and for any $(x,y)\in \bl{N}^{n\alpha}$, $x\in N^r_x(0)\cup N^r_x(1)$ and  $y\in N^s_y(0)\cup N^s_y(1)$ for some $0\leq r\leq n, 0\leq s\leq \alpha$. Thus, for $(x,y) \in N^{\alpha}$ we have

\begin{align}\label{peq}
f(x,y) = \begin{cases}
(r\!-\!1)(s\!-\!1)a& \text{if } x\in N^r_x(0), y\in N^s_y(0) \text{ for } (r,s)\in I^{\alpha}_n\cup I^n_{\alpha} \\
(ra\!-\!1)(s\!-\!1)& \text{if } x\in N^r_x(1), y\in N^s_y(0) \text{ for } (r,s)\in I^{\alpha}_n\cup I^n_{\alpha} , \\
(r\!-\!1)(sa\!-\!1)& \text{if } x\in N^r_x(0), y\in N^s_y(1) \text{ for } (r,s)\in I^{\alpha}_n\cup I^n_{\alpha} , \\
rsa\!-\!r\!-\!s\!+\!\lambda & \text{if } x\in N^r_x(1), y\in N^s_y(1) \text{ for } (r,s)\in I^{\alpha}_n\cup I^n_{\alpha}.
\end{cases}
\end{align}

where $I^p_q = \{0,1,\ldots,p\}\times \{0,1,\ldots,q\}$.\\
Thus, $(x^0,y^0)$ is locally optimal with respect to $N^{\alpha}$ if and only if
\begin{align}
(r\!-\!1)(s\!-\!1)a&\leq \lambda \text{ for all } (r,s)\in I^{\alpha}_n\cup I^n_{\alpha} \label{ll1}\\
(ra\!-\!1)(s\!-\!1)&\leq \lambda \text{ for all } (r,s)\in I^{\alpha}_n\cup I^n_{\alpha} , \label{ll2}\\
(r\!-\!1)(sa\!-\!1)&\leq \lambda \text{ for all } (r,s)\in I^{\alpha}_n\cup I^n_{\alpha} , \label{ll3}\\
rsa\!-\!r\!-\!s\! &\leq 0 \text{ for all } (r,s)\in I^{\alpha}_n\cup I^n_{\alpha}\label{ll4}.
\end{align}

It is not very difficult to verify that conditions (\ref{ll1}), (\ref{ll2}), (\ref{ll3}), and (\ref{ll4}), are satisfied if
\begin{align}
(\alpha-1)(n-1)a&\leq\lambda,\label{c1}\\
(n-1)(\alpha a-1)&\leq\lambda,\label{c2}\\
(\alpha-1)(na-1)&\leq\lambda,\text{ and}\label{c3}\\
a\alpha n- \alpha - n&\leq 0\label{c4}.
\end{align}
Choose $a=\frac{6}{n}$. Then inequality (\ref{c4}) holds  and inequality (\ref{c1}) implies inequalities (\ref{c2}) and (\ref{c3}). Choose $\lambda=(\alpha-1)(n-1)\frac{6}{n}$. Then $(x^0,y^0)$ is locally optimal. Now,

\begin{align*}\cl{A}(Q,c,d)-f(x^0,y^0)&=\frac{1}{4}((n-1)^2a-(2n-2)+\lambda)-\lambda\\
&=\frac{1}{4n}(\frac{2}{5}n^2+\frac{58}{5}n-12)> 0 \text{ for } n \geq 6.
\end{align*}
This completes the proof.
\end{proof}

As an immediate corollary, we have the following result.
\begin{corollary}For any fixed $h$ and $k$, the objective function value of a locally optimal solution with respect to the neighborhood $\bl{N}^{hk}$ could be worse than $\cl{A}(Q,c,d)$ for sufficiently large $m$ and $n$.
\end{corollary}

\begin{proof}
Choose $\alpha=\max\{h,k\}$. Then a locally optimal solution with respect to $\mathbb{N}^\alpha$ is not worse than a locally optimal with respect to $\bl{N}^{hk}$. The result now follows from Theorem~\ref{loc1}.
\end{proof}

These examples motivates us to develop polynomial time algorithms for BBQP that guarantee a solution with objective function value no worse than $\cl{A}(Q,c,d)$.

\section{Algorithms with no worse than average guarantee}

We first consider a very simple algorithm to compute a solution with objective function value guaranteed to be no worse than $\cl{A}(Q,c,d)$. The algorithm simply takes fractional vectors $x\in \bl{U}^n$ and $y\in \bl{U}^n$ and applies a rounding scheme to produce a solution for BBQP.
Let $x\in \bl{U}^m$ and $y\in \bl{U}^n$. Extending the definition of $f(x,y)$ for 0-1 vectors, define
$$f(x,y)=\sum_{i\in \cl{M}}\sum_{j\in \cl{N}}q_{ij}x_iy_j+\sum_{i\in \cl{M}}c_ix_i+\sum_{j\in \cl{N}}d_jy_j.$$
Consider the solutions $y^*\in \bl{B}^n$ and $x^*\in \bl{B}^m$ given by
\begin{equation}\label{eq1}
y^*_j = \begin{cases}
1 & \text{if } \displaystyle{d_j+\sum_{i\in \cl{M}} q_{ij}x_i   > 0}, \\
0 & \text{otherwise,}
\end{cases}
\end{equation}
\begin{equation}\label{eq2}
x^*_i = \begin{cases}
1 & \text{if } \displaystyle{c_i+\sum_{j\in \cl{N}} q_{ij} y^*_j  > 0}, \\
0 & \text{otherwise.}
\end{cases}
\end{equation}
Note that $x^*$ is the optimal 0-1 vector when $y$ is fixed at $y^*$, and equation (\ref{eq1}) rounds the $y$ to $y^*$ using a prescribed rounding criterion. The process of constructing $(x^*,y^*)$ from $(x,y)$ thus called a {\it round-$y$ optimize-$x$ algorithm} or RyOx-algorithm. The next theorem establishes a lower bound on the objective function value of the solution produced by the RyOx-algorithm.
\begin{theorem}\label{th3}$f(x^*,y^*) \geq f(x,y)$.\end{theorem}
\begin{proof}
\setlength{\abovedisplayskip}{0pt}
{\allowdisplaybreaks
\begin{align*}
f(x,y) & = \sum_{i\in \cl{M}}\sum_{j\in \cl{N}}q_{ij}x_iy_j+\sum_{i\in \cl{M}}c_ix_i+\sum_{j\in \cl{N}}d_jy_j\\ &=
 \left[\sum_{j\in \cl{N}}\left(\sum_{i\in \cl{M}}q_{ij}x_i+d_j\right )\right ]y_j+\sum_{i\in \cl{M}}c_ix_i\\
&\leq \sum_{j\in \cl{N}}\left(\sum_{i\in \cl{M}}q_{ij}x_i+d_j\right )y_j^*+\sum_{i\in \cl{M}}c_ix_i \tag{by construction of $y^*$}\\
&=\sum_{j\in \cl{N}}d_jy_j^* +\sum_{i\in \cl{M}} \left(\sum_{j\in \cl{N}} q_{ij}y_j^*+c_i\right )x_i\\
&\leq\sum_{j\in \cl{N}}d_jy_j^* +\sum_{i\in \cl{M}} \left(\sum_{j\in \cl{N}} q_{ij}y_j^*+c_i\right )x_i^* \tag{by construction of $x^*$}\\
&=f(x^*,y^*).\qedhere
\end{align*}
}
\end{proof}
Note that $(x^*,y^*)$ can be constructed in $O(mn)$ time whenever $x$ and $y$ are rational numbers. We can also round $x$ first to obtain $x^0\in \bl{B}^m$ and choose optimal $y=y^0$ by fixing $x$ at $x^0$.  This is done using the rounding scheme given by the following equations:
\begin{center}
\begin{tabular}{ccc}
$
x^0_i = \begin{cases}
1 & \text{if } \displaystyle{c_i+\sum_{j\in \cl{N}} q_{ij}y_j   > 0}, \\
0 & \text{otherwise,}
\end{cases}
$
& and &
$
y^0_j = \begin{cases}
1 & \text{if } \displaystyle{d_j+\sum_{i\in \cl{M}} q_{ij} x^0_i  > 0}, \\
0 & \text{otherwise,}
\end{cases}
$\\
\end{tabular}
\end{center}

\begin{theorem}\label{th3-1}$f(x^0,y^0) \geq f(x,y)$.\end{theorem}

The proof of Theorem~\ref{th3-1} follows along the same line as Theorem~\ref{th3} and hence omitted. The process of constructing $(x^0,y^0)$ is called {\it round-$x$ optimize-$y$ algorithm} or RxOy-algorithm. The complexity of the RxOy-algorithm is also $O(mn)$.

\begin{corollary}\label{cl1}
A solution $(\bar{x},\bar{y})$ for BBQP satisfying $f(\bar{x},\bar{y}) \geq \cl{A}(Q,c,d)$ can be obtained in $O(mn)$ time.
\end{corollary}
\begin{proof}
Let $x_i=1/2$ for all $i\in \cl{M}$ and $y_j=1/2$ for all $j\in \cl{N}$. Then it can be verified that $f(x,y)=\cl{A}(Q,c,d)$, but $(x,y)$ is not feasible for BBQP. Now, choose $(\bar{x},\bar{y})$ as the output of either the RyOx-algorithm or the RxOy-algorithm. The result now follows from Theorems~\ref{th3} or \ref{th3-1}.
\end{proof}

In view of Corollary~\ref{cl1} and Theorem~\ref{th3}, the dominance ratio of the RyOx-algorithm and the RxOy-algorithm is at least $\frac{1}{4}$. By choosing an appropriate starting solution, we can establish improved dominance ratio for these algorithms.

We now discuss an unexpected upper bound on $\cl{A}(Q,c,d)$. As a consequence, we have yet another simple scheme to compute a solution that is not worse than average. Let $\alpha = \sum_{i\in \cl{M}}\sum_{j\in \cl{N}}q_{ij}$, $\beta=\sum_{i\in \cl{M}}c_i$ and $\gamma = \sum_{j\in \cl{N}}d_j$.
\begin{theorem}\label{th4}$\cl{A}(Q,c,d)\leq \max\{\alpha+\beta+\gamma, \beta, \gamma, 0\}$.\end{theorem}
\begin{proof}
Let $u$ and $v$ be real numbers in $[0,1]$. Choose $x\in \bl{B}^m$, $y\in \bl{B}^n$ be such that $x_i=u$ for all $i\in \cl{M}$ and $y_j=v$ for all $j\in \cl{N}$. Then $f(x,y)=\alpha uv+\beta u +\gamma v=\eta(u,v),$ say. Note that $\eta(1/2,1/2)=\cl{A}(Q,c,d)$. Thus, $\max\{\eta(u,v) : (u,v)\in \bl{U}^2\} \geq \cl{A}(Q,c,d)$. Since $\eta(u,v)$ is bilinear, its maximum is attained at an extreme point of the square $\bl{U}^2$. Since these extreme points are precisely $(0,0), (1,0), (0,1), (1,1)$, the result follows.
\end{proof}

\begin{corollary}\label{cr2}
One of the solutions $({\bd{1}}^m,{\bd{1}}^n),({\bd{1}}^m,{\bd{0}}^n),({\bd{0}}^m,{\bd{1}}^n),({\bd{0}}^m,{\bd{0}}^n)$ of BBQP  have an objective function value no worse than $\cl{A}(Q,c,d)$.
\end{corollary}
The proof of this corollary follows directly from Theorem~\ref{th4}. We can compute $\alpha$, $\beta$ and $\gamma$ in $O(mn)$ time and hence we have a solution no worse than average in $O(mn)$ time. Interestingly, if $\alpha, \beta$ and $\gamma$ are given, then we can identify  a solution to BBQP with objective function value no worse than $\cl{A}(Q,c,d)$ in $O(1)$ time. The solution produced by Corollary~\ref{cr2} is trivial and may not be of much practical value. Nevertheless, the simple upper bound on $\cl{A}(Q,c,d)$ established by Theorem~\ref{th4} is very interesting and have interesting consequences as discussed below.

 Recall that the alternating algorithm starts with a solution $(x^0,y^0)$, fix $x^0$ and find the best $y$, say $y^1$. Then fix $y$ at $y^1$ and compute the optimal $x$ and so on. Since we initiate the algorithm by fixing $x$ first, we call this the \emph{$x$-first alternating algorithm}. We can also start the algorithm by fixing $y^0$ first and the resulting variation is called the \emph{$y$-first alternating algorithm}.

 \begin{theorem}\label{xth1}The best solution amongst the solutions produced by the $x$-first alternating algorithm with starting solutions $({\bd{1}}^m,{\bd{1}}^n)$ and ${(\bd{0}}^m,{\bd{0}}^n)$ dominates $2^{m+n-2}+3\left(2^{n-1}\right )$ solutions.
 \end{theorem}
 \begin{proof}
 Let $(x^*,y^*)$ be the best solution obtained. When the starting solution is ${(\bd{1}}^m,{\bd{1}}^n)$, $f(x^*,y^*) \geq \max\{f({\bd{1}}^m,{\bd{1}}^n), f({\bd{1}}^m,{\bd{0}}^n)\}$. Likewise, when the starting solution is ${(\bd{0}}^m,{\bd{0}}^n)$, we have $f(x^*,y^*) \geq \max\{f({\bd{0}}^m,{\bd{0}}^n), f({\bd{0}}^m,{\bd{1}}^n)\}$. Thus, by Theorem~\ref{th4} $f(x^*,y^*) \geq \cl{A}(Q,c,d)$ and hence by Theorem~\ref{dth1}, $(x^*,y^*)$ dominates at least $2^{m+n-2}$ solutions. To account for the remaining solutions that are dominated by $(x^*,y^*)$ we proceed as follows.\\

 Construct the set $D_1$ as in the proof of Theorem~\ref{dth1}. Recall that $|D_1|=2^{m+n-2}$. Let $D_2=\{({\bd{1}}^m,y), ({\bd{0}}^m,y) : y\in \bl{B}^n\}$. Now, $|D_2|=2^{n+1}$. For any $y\in \bl{B}^n$, we have $P({\bd{0}}^m,y)=P({\bd{1}}^m,y)$ and $P({\bd{0}}^m,y)=\{({\bd{0}}^m,y), ({\bd{1}}^m,y), ({\bd{0}}^m,{\bd{1}}^n-y), ({\bd{1}}^m,{\bd{1}}^n-y)\}$. Thus, $D_2$ can be partitioned into sets of the form $P({\bd{0}}^m,y^k)$ and there are $\frac{1}{4}2^{n+1}=2^{n-1}$ such sets. Exactly one element from each $P({\bd{0}}^m,y^k)$ is in $D_1$. Thus, $|D_1\cap D_2|=2^{n-1}$ and hence $|D_1\cup D_2|=2^{m+n-2}+2^{n+1}-2^{n-1}=2^{n-1}\left(2^{m-1}+3\right )$.
 \end{proof}

From Theorem~\ref{xth1} the dominance ratio of the $x$-first alternating algorithm is at least $\frac{1}{4}+\frac{3}{2^{m+1}}$ when starting solutions are selected carefully and the algorithm is applied twice. Note that to achieve this dominance ratio, we simply need to perform only one iteration each, when the algorithm starts with $({\bd{1}}^m,{\bd{1}}^n)$ and ${(\bd{0}}^m,{\bd{0}}^n)$. Thus, we can achieve this dominance ratio in polynomial time. A similar result can be derived for $y$-first alternating algorithm and local search algorithm with neighborhood $\bl{N}^\alpha$ for any $\alpha \geq 0$. Further, similar dominance ratio can be achieved by RxOy-algorithm or RyOx-algorithm applied twice, once starting with $({\bd{1}}^m,{\bd{1}}^n)$ and then starting with ${(\bd{0}}^m,{\bd{0}}^n)$ and choosing the best solution.\\

The problem BBQP can be formulated as integer linear programming problem~\cite{p1} as follows:
\begin{align*}
\text{ILP1: } \text{Maximize} & \sum_{i\in \cl{M}}\sum_{j\in \cl{N}}q_{ij}z_{ij}+\sum_{i\in \cl{M}}c_ix_i+\sum_{j\in \cl{N}}d_jy_j \\
\text{Subject to } &z_{ij}-x_i \leq 0, i\in \cl{M}, j\in \cl{N}\\
&z_{ij}-y_j \leq 0, i\in \cl{M}, j\in \cl{N}\\
&z_{ij}-x_i-y_j\geq -1, (i,j)\in S,\\
& x_i \in \{0,1\} \mbox{ for } i \in \cl{M}, y_j \in \{0,1\} \mbox{ for } j\in \cl{N}, z_{ij} \in \{0,1\} \mbox{ for } (i,j) \in \cl{M}\times \cl{N},
\end{align*}
where $S=\{ij : q_{ij} < 0\}$. Let $(z^{\prime},x^{\prime},y^{\prime})$ be an optimal solution to the LP relaxation, where $z^{\prime}$ is an $m\times n$ matrix with $(i,j)$th entry $z^{\prime}_{ij}$, and $x^{\prime}$ is an $m$-vector  with $i$th entry $x^{\prime}_i$ and $y^{\prime}$ is an $n$-vector with $j$th entry $y^{\prime}_j$.  Recall that for $x^{\prime}\in [0,1]^m$ and $y^{\prime}\in [0,1]^n$, $\phi(x^{\prime},y^{\prime})=\sum_{i\in \cl{M}}\sum_{j\in \cl{N}}q_{ij}x^{\prime}_iy^{\prime}_j+\sum_{i\in \cl{M}}c_ix^{\prime}_i+\sum_{j\in \cl{N}}d_jy^{\prime}_j$ and $h(z^{\prime},x^{\prime},y^{\prime})$ is the optimal objective function value of the linear programming relaxation of ILP1.

For any $x\in [0,1]^m$ and $y\in [0,1]^n$, the  solution $(x^1,y^1)$ obtained by the RxOy (RyOx) algorithm satisfies $\phi(x,y)\leq f(x^1,y^1) \leq h(z^{\prime},x^{\prime},y^{\prime})$. This follows from Theorem~\ref{th3} and the property of LP relaxations. Alternative integer programming formulations of BBQP could provide different LP relaxation solutions and hence the resulting RxOy (RyOx) rounding solutions could be different. We give below a new integer programming formulation of BBQP by increasing the number of variables.
\begin{align*}
\text{ILP2: } \text{Maximize} & \sum_{i\in \cl{M}}\sum_{j\in \cl{N}}q_{ij}\left (\frac{1}{4}u_{ij}+v_{ij} -\frac{1}{4}w_{ij}-\frac{1}{4}z_{ij}\right )+\sum_{i\in \cl{M}}c_ix_i+\sum_{j\in \cl{N}}d_jy_j \\
\text{Subject to } &u_{ij}+2v_{ij}=x_i+y_j \mbox{ for } i\in \cl{M}, j\in \cl{N}\\
&u_{ij}+v_{ij}\leq 1, \mbox{ for } i\in \cl{M}, j\in \cl{N}\\
&-z_{ij}+ w_{ij} = x_i-y_j \mbox{ for } i\in \cl{M}, j\in \cl{N}\\
&z_{ij}+w_{ij} \leq 1 \mbox{ for } i\in \cl{M}, j\in \cl{N}\\
& x_i \in \{0,1\} \mbox{ for } i \in \cl{M}, y_j \in \{0,1\} \mbox{ for } j\in \cl{N}\\
& u_{ij},v_{ij},w_{ij},z_{ij} \in \{0,1\} \mbox{ for } (i,j) \in \cl{M}\times \cl{N},
\end{align*}

Experimental analysis of RxOy (RyOx) algorithms starting from the LP relaxations of ILP1 and ILP2 are discussed in the next section.

\section{Computational results}

 Although the primary focus of the paper is on theoretical analysis of approximation algorithms, we have conducted preliminary experimental analysis with the RxOy and RyOx rounding algorithms to examine features that are not clear from theoretical analysis. The alternating algorithm and local search algorithms are thoroughly analyzed from an experimental analysis point of view in~\cite{kp} and~\cite{x} and hence are not considered in our experimental study. The rounding algorithms being simple constructive heuristics, they are not expected to outperform more sophisticated algorithms that employ powerful neighborhoods within a metaheuristic framework. Nonetheless, it is interesting to explore the behavior of these algorithms in the light its theoretical properties and the potential for computing initial solutions for more advanced solution improvement algorithms.

  The data set used in our experiments  are smallsize instances of random problems, biclique instances, matrix factorization instances, and maxcut instances from~\cite{kp}. The algorithms are implemented in C\# and tested on a DELL PC with Windows 7 operating system, Intel i7 processor and 16GB of memory. All CPU times reported are in milliseconds and do not include input-output times. We used the solution obtained by the linear programming relaxation of ILP1 as the fractional vector $(x,y)$ to initiate the rounding process.

 Our first set of experiments were aimed to identify the percentage of problems where LP relaxation produced optimal solutions. This data is important since for such problems, rounding (heuristic algorithms) is irrelevant.  The results of these experiments are summarized in Figure 1. The trend shows that as the problem size increases, the number of problems where the LP relaxation produced optimal solution decreases. This decrease is more rapid in some class of problems compared to random instances.

\begin{figure}
\centering
\begin{minipage}{.45\textwidth}
  \centering
\begin{tikzpicture}
	\begin{axis}[
		width=\textwidth,
		height=5cm,
		legend pos=outer north east,
		xlabel={$mn$},
		ylabel={Opt., \%},
		cycle list={
			{black,only marks},
		},
	]
	\addplot+ coordinates {
		(16, 96)
		(20, 94)
		(25, 92)
		(24, 91)
		(30, 91)
		(36, 87)
		(28, 88)
		(35, 90)
		(42, 87)
		(49, 79)
		(32, 91)
		(40, 81)
		(48, 79)
		(56, 72)
		(64, 59)
		(36, 85)
		(45, 83)
		(54, 70)
		(63, 70)
		(72, 55)
		(81, 43)
		(40, 89)
		(50, 81)
		(60, 80)
		(70, 61)
		(80, 45)
		(90, 45)
		(100, 36)
		(44, 83)
		(55, 77)
		(66, 64)
		(77, 57)
		(88, 45)
		(99, 39)
		(110, 39)
		(121, 22)
		(48, 82)
		(60, 78)
		(72, 67)
		(84, 60)
		(96, 39)
		(108, 33)
		(120, 25)
		(132, 22)
		(144, 20)
		(52, 81)
		(65, 67)
		(78, 60)
		(91, 50)
		(104, 46)
		(117, 32)
		(130, 23)
		(143, 12)
		(156, 15)
		(169, 17)
		(56, 80)
		(70, 70)
		(84, 57)
		(98, 44)
		(112, 33)
		(126, 27)
		(140, 26)
		(154, 15)
		(168, 13)
		(182, 6)
		(196, 4)
		(60, 81)
		(75, 63)
		(90, 50)
		(105, 40)
		(120, 34)
		(135, 18)
		(150, 18)
		(165, 13)
		(180, 8)
		(195, 3)
		(210, 5)
		(225, 2)
	};
	\end{axis}
\end{tikzpicture}
\caption{Random}
\end{minipage}%
\begin{minipage}{.1\textwidth}
\end{minipage}
\begin{minipage}{.45\textwidth}
  \centering
\begin{tikzpicture}
	\begin{axis}[
		width=\textwidth,
		height=5cm,
		xlabel={$mn$},
		ylabel={Opt., \%},
		cycle list={
			{black,only marks},
		},
	]
	\addplot+ coordinates {
		(16, 74)
		(20, 66)
		(25, 53)
		(24, 55)
		(30, 50)
		(36, 39)
		(28, 47)
		(35, 40)
		(42, 35)
		(49, 31)
		(32, 42)
		(40, 33)
		(48, 22)
		(56, 14)
		(64, 12)
		(36, 53)
		(45, 28)
		(54, 23)
		(63, 13)
		(72, 10)
		(81, 6)
		(40, 50)
		(50, 28)
		(60, 17)
		(70, 13)
		(80, 3)
		(90, 2)
		(100, 2)
		(44, 45)
		(55, 22)
		(66, 17)
		(77, 11)
		(88, 3)
		(99, 3)
		(110, 4)
		(121, 2)
		(48, 42)
		(60, 19)
		(72, 10)
		(84, 5)
		(96, 3)
		(108, 1)
		(120, 2)
		(132, 1)
		(144, 0)
		(52, 46)
		(65, 22)
		(78, 8)
		(91, 4)
		(104, 1)
		(117, 2)
		(130, 1)
		(143, 1)
		(156, 1)
		(169, 0)
		(56, 34)
		(70, 18)
		(84, 5)
		(98, 2)
		(112, 2)
		(126, 2)
		(140, 1)
		(154, 0)
		(168, 0)
		(182, 0)
		(196, 0)
		(60, 39)
		(75, 14)
		(90, 8)
		(105, 2)
		(120, 2)
		(135, 1)
		(150, 0)
		(165, 0)
		(180, 0)
		(195, 0)
		(210, 0)
		(225, 0)
	};
	\end{axis}
\end{tikzpicture}
\caption{Matrix Factorization}
\end{minipage}
\end{figure}

Let us now discuss the experimental results with the RxOy and RyOx algorithms. The details are summarized in Tables 1 to 5. In each table, the column ``best'' reports the best known objective function value of the problem. Most of these values are optimal. The column ``LP obj'' contains the upper bound obtained by the LP relaxation. The columns ``xy'' and ``yx'' respectively contains the objective function value of the solution produced by the RxOy algorithm and the RyOx algorithms.  It may be noted that the solution produced by these algorithms are close to the best known solution values and the running time is very minimal. This makes these algorithms good candidates to  be used to generate starting solutions for more sophisticated algorithms. The table also provide insight into the value $\cl{A}(Q,c,d)$ in comparison to the best known objective function value. The column ``avg'' reports $\cl{A}(Q,c,d)$ and the column ``Avg+'' reports the value of the upper bound on $\cl{A}(Q,c,d)$ provided by Theorem~\ref{th4}. Note that Avg+ is the objective function value of the solution reported in Corollary~\ref{cr2}. These are certainly inferior solutions. The column ``$x_iy_j$'' represents $f(x,y)$ for the LP relaxation solution $(x,y)$. The experimental results conclusively demonstrate the power of RxOy and RyOx algorithms as a very fast stand alone heuristics or  as an algorithm for generating starting solutions in more complex algorithms.

\begin{table}[ht] \centering
\begin{tabular}{@{} l @{} c @{} r r r r r r r @{} c @{} r r r @{} c @{} r @{}}
\toprule
&&\multicolumn{7}{c}{Objective}&&\multicolumn{3}{c}{Time, ms}&&\\
\cmidrule(){3-9}
\cmidrule(){11-13}
Instance&\hspace*{2em}&Best&LP obj.&$x_i y_j$&yx&xy&Avg+&Avg&\hspace*{2em}&LP&yx&xy\\
\midrule
$20 \times 50$&&13555&21179&437.5&12715&12208&2160&624.0&&151&171&146\\
$25 \times 50$&&13207&25405&759.0&12226&11016&2911&505.0&&69&77&87\\
$30 \times 50$&&15854&29713&634.0&13992&14824&3219&778.0&&102&105&102\\
$35 \times 50$&&14136&33471&-2146.0&12554&12099&835&-2028.5&&137&143&131\\
$40 \times 50$&&18778&39444&42.8&16562&15573&624&125.8&&168&183&170\\
$45 \times 50$&&22057&44760&-401.0&17688&19660&610&-440.0&&212&236&200\\
$50 \times 50$&&23801&50390&1576.0&21645&22178&5115&1576.0&&234&272&233\\
\midrule
Average&&17341&34909&128.9&15340&15365&2211&162.9&&154&169&153\\
\bottomrule
\end{tabular}
\caption{Random}
\label{}
\end{table}

\begin{table}[ht] \centering
\begin{tabular}{@{} l @{} c @{} r r r r r r r @{} c @{} r r r @{} c @{} r @{}}
\toprule
&&\multicolumn{7}{c}{Objective}&&\multicolumn{3}{c}{Time, ms}&&\\
\cmidrule(){3-9}
\cmidrule(){11-13}
Instance&\hspace*{2em}&Best&LP obj.&$x_i y_j$&yx&xy&Avg+&Avg&\hspace*{2em}&LP&yx&xy\\
\midrule
$20 \times 50$&&18341&29999&-1911363.0&4933&0&0&-1911363.0&&47&49&47\\
$25 \times 50$&&24937&38513&-3265058.5&2764&0&0&-3265058.5&&66&90&65\\
$30 \times 50$&&27887&51726&-2981174.8&9050&15138&0&-2981174.8&&97&95&95\\
$35 \times 50$&&32515&57302&-3782013.0&3270&10502&0&-3782013.0&&135&134&134\\
$40 \times 50$&&33027&61966&-4562391.5&0&0&0&-4562391.5&&133&131&133\\
$45 \times 50$&&37774&67923&-5789225.5&4420&0&0&-5789225.5&&140&141&140\\
$50 \times 50$&&30124&78745&-5004969.5&5517&5833&0&-5004969.5&&200&195&200\\
\midrule
Average&&29229&55168&-3899456.5&4279&4496&0&-3899456.5&&117&119&116&&0\\
\bottomrule
\end{tabular}
\caption{Max Biclique}
\label{}
\end{table}

\begin{table}[ht] \centering
\begin{tabular}{@{} l @{} c @{} r r r r r r r @{} c @{} r r r @{} c @{} r @{}}
\toprule
&&\multicolumn{7}{c}{Objective}&&\multicolumn{3}{c}{Time, ms}&&\\
\cmidrule(){3-9}
\cmidrule(){11-13}
Instance&\hspace*{2em}&Best&LP obj.&$x_i y_j$&yx&xy&Avg+&Avg&\hspace*{2em}&LP&yx&xy\\
\midrule
$20 \times 50$&&6983&10704&-315.3&5223&5481&0&-393.3&&110&109&111\\
$25 \times 50$&&8275&13866&-678.5&7630&7350&0&-678.5&&169&166&164\\
$30 \times 50$&&10227&18958&-140.0&8378&8685&0&-140.0&&232&234&228\\
$35 \times 50$&&11897&20590&-543.8&11071&9008&0&-543.8&&283&283&277\\
$40 \times 50$&&14459&23110&892.3&13216&13433&3569&892.3&&506&524&498\\
$45 \times 50$&&13247&24393&-984.5&12496&9513&0&-984.5&&430&431&428\\
$50 \times 50$&&15900&28875&356.3&14075&15357&1425&356.3&&628&622&621\\
\midrule
Average&&11570&20071&-201.9&10298&9832&713&-213.1&&337&338&333\\
\bottomrule
\end{tabular}
\caption{Max Induced Subgraph}
\label{}
\end{table}

\begin{table}[ht] \centering
\begin{tabular}{@{} l @{} c @{} r r r r r r r @{} c @{} r r r @{} c @{} r @{}}
\toprule
&&\multicolumn{7}{c}{Objective}&&\multicolumn{3}{c}{Time, ms}&&\\
\cmidrule(){3-9}
\cmidrule(){11-13}
Instance&\hspace*{2em}&Best&LP obj.&$x_i y_j$&yx&xy&Avg+&Avg&\hspace*{2em}&LP&yx&xy\\
\midrule
$20 \times 50$&&9008&21266&-791.0&3838&4928&0&-780.0&&93&83&89\\
$25 \times 50$&&10180&27546&-1205.0&4062&6410&0&-1352.0&&123&130&120\\
$30 \times 50$&&13592&37676&-1172.0&5808&7560&0&-274.0&&205&184&183\\
$35 \times 50$&&14024&40900&-1196.0&5816&8578&0&-1105.0&&331&267&269\\
$40 \times 50$&&17610&45948&1133.0&9424&12082&3568&1784.0&&334&290&297\\
$45 \times 50$&&15252&48492&-1209.0&6204&9790&0&-1956.0&&441&402&402\\
$50 \times 50$&&19580&57408&723.0&12332&10364&1446&723.0&&591&538&536\\
\midrule
Average&&14178&39891&-531.0&6783&8530&716&-422.9&&303&271&271\\
\bottomrule
\end{tabular}
\caption{MaxCut}
\label{}
\end{table}

\begin{table}[ht] \centering
\begin{tabular}{@{} l @{} c @{} r r r r r r r @{} c @{} r r r @{} c @{} r @{}}
\toprule
&&\multicolumn{7}{c}{Objective}&&\multicolumn{3}{c}{Time, ms}&&\\
\cmidrule(){3-9}
\cmidrule(){11-13}
Instance&\hspace*{2em}&Best&LP obj.&$x_i y_j$&yx&xy&Avg+&Avg&\hspace*{2em}&LP&yx&xy\\
\midrule
$20 \times 50$&&114&249&-1.5&106&94&0&-1.5&&37&32&31\\
$25 \times 50$&&127&314&1.5&109&96&6&1.5&&45&40&42\\
$30 \times 50$&&148&373&-2.5&124&140&0&-2.5&&65&61&61\\
$35 \times 50$&&139&425&-13.0&106&122&0&-13.0&&80&77&78\\
$40 \times 50$&&210&512&12.0&182&194&48&12.0&&148&73&71\\
$45 \times 50$&&191&559&-3.5&165&179&0&-3.5&&100&93&94\\
$50 \times 50$&&217&637&11.5&181&191&46&11.5&&120&119&119\\
\midrule
Average&&164&438&0.6&139&145&14&0.6&&85&71&71\\
\bottomrule
\end{tabular}
\caption{Matrix Factorization}
\label{}
\end{table}
\clearpage
We have also conducted experiments with RxOy and RyOx rounding algorithms using the LP relaxation solution of ILP2. For random instances, the LP relaxation solution of ILP1 and ILP2 were different but for all other test instances, they produced same solutions most of the time. In some random instances considered, the rounding algorithms produced better solutions when started ILP2 LP relaxation solution in comparison to the solutions produced from ILP1 LP relaxation. However, in general we did not see significant performance difference and hence preferred ILP1 formulation, considering its smaller size.\\

Note that the RxOy and RyOx rounding algorithms can be initiated using any $x^0\in \bl{U}^m$ and $y^0\in \bl{U}^n$ and not necessarily a fractional solution for an LP relaxation.
For $i=1,2,\ldots ,m$ let $\gamma_i =c_i+ \sum_{j\in \cl{N}}q_{ij}$ and for $j=1,2,\ldots ,n$ let $\delta_j=d_j+\sum_{i\in \cl{M}}q_{ij}$. Also, $ran$[$a,b$] represents a uniformly distributed random number in the interval [$a,b$]. Using these values we generate random vectors $(x^0,y^0)$ as follows:

\noindent{\bf Type 1 random vectors:} Choose
\begin{center}
\begin{tabular}{ccc}
 $ x^0_i = \begin{cases}
\text{ran($0,0.5$]} & \text{if } \gamma_i < 0, \\
\text{ran($0.5,1$]} & \text{otherwise,}
\end{cases}
$
& and &
$ y^0_j = \begin{cases}
\text{ran($0,0.5$]} & \text{if } \delta_j  < 0, \\
\text{ran($0.5,1$]} & \text{otherwise.}
\end{cases}$
\end{tabular}
\end{center}

\noindent{\bf Type 2 random vectors:} These are generated using weighted random numbers with weight proportional to the corresponding $\gamma_i$ or $\delta_j$ values. Let  $\beta_1 = \max\{|\gamma_i| : i\in M\}$ and $\beta_2 = \max\{|\delta_j| : j\in \cl{N}\}$. Now choose
\begin{center}
\begin{tabular}{ccc}
 $ x^0_i = \begin{cases}
\text{$0.5-\dfrac{|\gamma_i|}{\beta_1}\text{ran}$($0,0.5$]} & \text{if } \gamma_i < 0, \\
\text{$\dfrac{\gamma_i}{\beta_1}\text{ran}$($0.5,1$]} & \text{otherwise,}
\end{cases}
$
& and &
$ y^0_j = \begin{cases}
\text{$0.5-\dfrac{|\delta_j|}{\beta_2}ran$($0,0.5$]} & \text{if } \delta_j  < 0, \\
\text{$\dfrac{\delta_j}{\beta_2}ran$($0.5,1$]} & \text{otherwise.}
\end{cases}$
\end{tabular}
\end{center}

The RxOy and RyOx algorithms using the above choices of starting solution $(x^0,y^0)$ can also be used to construct starting solutions for advanced algorithms. The built-in randomness generates good solutions that can be embedded in metaheristics with multiple starts. Systematic experimental analysis of such sophisticated algorithms is beyond the scope of this paper.

\section{Conclusion}

In this paper we studied approximation algorithms for BBQP which is a generalization of the well known BQP. Various approximation algorithms are analyzed using averaged value based measures and domination analysis.   It is demonstrated that very powerful local search algorithms could get trapped at poor quality local minimum even if we allow exponential time in searching a very large scale neighborhood. Some of the proof techniques used are simple yet elegant and could be of use in domination analysis of heuristics for other related problems. Experimental results with two construction algorithms are also given. A natural question for further investigation is to close the gap between non-approximability bounds and lower bounds on domination ratio. Since BBQP is not as well studied as BQP, there are many other avenues for further investigation and we are currently investigating further properties of BBQP.

\end{document}